\documentclass[preprint,11pt]{elsarticle}

\usepackage{amssymb,amsmath,amsthm}
\usepackage{graphicx}
\usepackage{float}
\restylefloat{figure}

\newcommand{\beq}[1]{ \begin{equation}\label{#1} }
\newcommand{\eeq}{\end{equation}}
\numberwithin{equation}{section}
\DeclareMathOperator*{\esssup}{ess\,sup}

\def\RR{{\mathbb R}}

\usepackage{fullpage}
\newtheorem{theorem}{Theorem}
\newtheorem{lemma}{Lemma}
\newtheorem{corollary}{Corollary}
\newtheorem{example}{Example}
\newtheorem{remark}{Remark}
\journal{Applied Mathematics and Computation}

\begin{document}

\begin{frontmatter}

\title{Solution estimates for linear differential equations with delay}

\author[label1]{Leonid Berezansky}
\author[label2]{Elena Braverman}
\address[label1]{Dept. of Math.,
Ben-Gurion University of the Negev,
Beer-Sheva 84105, Israel}
\address[label2]{Dept. of Math. and Stats., University of
Calgary,2500 University Drive N.W., Calgary, AB, Canada T2N 1N4; e-mail
maelena@ucalgary.ca, phone 1-(403)-220-3956, fax 1-(403)--282-5150 (corresponding author)}






\begin{abstract}
In this paper, we give explicit exponential estimates
$\displaystyle
|x(t)|\leq M e^{ -\gamma (t-t_0) }$, where $t\geq t_0$, $M>0$,
for solutions of a linear scalar delay differential equation  
$$
\dot{x}(t)+\sum_{k=1}^m b_k(t)x(h_k(t))=f(t),~~ t\geq t_0,~
x(t)=\phi(t),~t\leq t_0.
$$
We consider two different cases:  when $\gamma>0$ (corresponding to exponential stability) and the case of
$\gamma <0$ when the solution is, generally, growing.

In the first case, together with the exponential estimate, we also obtain an 
exponential stability test, in the second case we get estimation for solution growth. 
Here both the coefficients and the delays are measurable, not necessarily continuous. 
\end{abstract}



\end{frontmatter}

\section{Introduction}

Exponential or asymptotic stability of solutions is one of the most important properties of a functional differential equation (FDE).
There are many publications on stability of a FDE.
We cite here the papers \cite{BB2,BB3,BB5,BB1,GD,GH1,GH2,Kz,SYC,YS} and the monographs 
\cite{AS,F,Gil,Kharitonov,KM} 
fully  or partially devoted to  asymptotic stability of this class of equations. 
Asymptotic stability describes long time behavior of solutions.  
But in applications of FDEs, usually it is necessary to know 
estimates of solutions  on finite intervals. 
The following well-known topics are close to one considered in the paper: Lyapunov 
exponents, asymptotic integration, and growth rates 
\cite{Appleby,Appleby2,Barreira,Diblik1,Diblik2,Dom,Gyori,Khar,Philos,Pinto,Pituk,Hu}.

In this paper, we propose exponential estimates for solutions of a scalar delay differential equation (DDE) with variable coefficients and delays
\begin{equation}\label{1.1}
\dot{x}(t)+\sum_{k=1}^m b_k(t)x(h_k(t))=0,~~t\geq t_0,~ h_k(t)\leq t, ~k=1, \dots,m. 
 \end{equation}
Estimates for non-homogeneous equations are also considered.
We investigate two different cases, one when the equation is exponentially stable and the second  when it is not,
or when asymptotic behaviour of solutions is unknown. 

For the equation with one delay
$$
\dot{x}(t)+a(t)x(h(t))=0,
$$
for example, the first case occurs when $a(t)\geq a_0>0$ and the equation is exponentially stable, the second case is for 
$a(t)\geq 0$
but the equation is not exponentially stable,
or for $a(t)$ being non-positive or oscillatory.

All our conditions can be applied to a wide class of scalar linear 
differential equations with variable coefficients and delays, without the assumption that these parameters are continuous functions, leading to solution estimates.
For equations with variable coefficients and delays,
such solution estimates are obtained for the first time, and for  
ordinary differential equations are sharp.
                             
The paper is organized as follows. After some definitions and auxiliary results in Section 2, Section 3 contains an estimate 
for the fundamental function of an  equation with
a non-delay term. Section 4 is the main part of the paper. Here we consider solution estimates 
in the two above mentioned cases. 
Section 5 presents illustrating examples. Section 6 contains a discussion and suggests some projects for future research.

\section{Preliminaries}

We consider equation  (\ref{1.1}) under the following conditions:

(a1)  $b_k:[0,\infty) \to {\mathbb R}$ are Lebesgue measurable essentially bounded 
functions; 

(a2)  the functions $h_k$  are Lebesgue measurable  on $[0,\infty)$, 
and 
for some finite constants 
$\delta_k\leq\tau_k$,
$0\leq \delta_k\leq t-h_k(t)\leq \tau_k$, $k=1, \dots, m$, $t \geq t_0 \geq 0$, $\tau=\max_k \tau_k$.

Together with equation (\ref{1.1}), we consider for any $t_0\geq 0$ an initial value problem for the non-homogeneous
equation 
\begin{equation}\label{2.1}
\dot{x}(t)+\sum_{k=1}^m b_k(t)x(h_k(t))=f(t),~ t\geq t_0; ~~x(t)=\phi(t),~ t\leq t_0,
\end{equation}
where 
\\
(a3) $f:[t_0,\infty) \to {\mathbb R}$ is a Lebesgue measurable  locally essentially bounded 
function,  the initial function $\phi:[t_0-\tau, t_0] \rightarrow \RR$ is bounded and Borel measurable.


By {\bf the solution of problem} (\ref{2.1}) we mean a locally  absolutely continuous function  $x: [t_0,\infty)\to {\mathbb R}$ satisfying almost everywhere the equation, whenever
$t\geq t_0$,
and the initial conditions if $t\leq t_0$.

For any fixed $s$, the solution $X(t,s)$ of the problem with the zero initial function and the initial value 
being equal to one
$$
\frac{\partial X(t,s)}{\partial t}+\sum_{k=1}^m b_k(t)X(h_k(t),s)=0, ~t\geq s; ~~X(t,s)=0, ~t<s,~~ X(s,s)=1
$$
is called {\bf the fundamental function}.



According to \cite[Theorem 4.3.1]{AS},  the solution of  (\ref{2.1}) exists and is unique.
Also, the solution has the representation 
\begin{equation}
\label{2.2}
x(t)=X(t,t_0)x(t_0)-\sum_{k=1}^m \int_{t_0}^{t_0+\tau_k} X(t,s) b_k(s)\phi(h_k(s))ds+\int_{t_0}^t X(t,s)f(s)ds.
\end{equation}
In (\ref{2.2}), we assume $\phi(t)=0$, $t>t_0$.

We call equation (\ref{1.1})  {\bf uniformly exponentially stable} 
if, for some positive numbers $M$ and $\gamma$, not depending on $t_0$ and $\phi$,
a solution of homogeneous (\ref{2.1}), where $f\equiv 0$, satisfies 
\begin{equation}\label{2.3}
|x(t)|\leq M e^{-\gamma (t-t_0)} \sup_{t \in [t_0 - \tau,  t_0]}|\phi(t)|,~~t\geq t_0.
\end{equation}

We will say that the fundamental function $X(t,s)$ of (\ref{1.1}) 
{\bf has an exponential estimate} if there exist $M_0>0$ and $\gamma_0>0$ such that
\begin{equation}\label{2.4}
|X(t,s)|\leq M_0 e^{-\gamma_0(t-s)}, \mbox{~~for any~~~} t\geq s\geq 0.
\end{equation}

We will further apply the Bohl-Perron theorem stated below.

\begin{lemma}\label{lemmaBP}\cite[Theorem 4.7.1]{AS}
Assume that (a1)-(a2) hold. If the solution of the problem 
\begin{equation}\label{10}
\dot{x}(t)+\sum_{k=1}^m b_k(t)x(h_k(t))=f(t), ~x(t)=0,~t\leq t_0   
\end{equation}
is bounded on $[t_0,\infty)$ for any 
essentially bounded function $f:[t_0,\infty) \to {\mathbb R}$, 
equation (\ref{1.1}) is uniformly exponentially stable.
\end{lemma}

The main objective of the present paper is to obtain explicit estimates (\ref{2.3})-(\ref{2.4}) and their generalizations for homogeneous equation (\ref{1.1}). 
We consider the case when the equation is exponentially stable and when it is unstable, or asymptotic behaviour is unknown.
In the former case $\gamma>0$, while in the latter case $\gamma<0$.
We also get estimates for  non-homogeneous equations (\ref{2.1}).

\section{Estimates 
for an Auxiliary Equation}

In this section, for  the fundamental function $Y(t,s)$ of the  equation involving both delay terms and a non-delay term
\begin{equation}\label{3.2}
\dot{y}(t)=c(t)y(t)-\sum_{k=1}^m d_k(t)y(h_k(t)), ~t \geq t_0
\end{equation}
we deduce a uniform estimate
\begin{equation}\label{3.1}
|Y(t,s)|\leq K \mbox{~~for~~} t \geq s \geq t_0.
\end{equation}

We assume that for $c, d_k$, assumption (a1) holds,
while the delay functions $h_k$ satisfy (a2).

Denote $\|f\|_J=\esssup_{t\in J}|f(t)|$, where $J=[t_0,t_1]$ or $J=[t_0,\infty)$, $\displaystyle d(t):=\sum_{k=1}^m d_k(t)$.

\begin{lemma}\label{lemma3.1}
If there is an $\alpha_0 >0$ such that
\begin{equation}\label{3.2a}
d(t)-c(t)\geq \alpha_0,~~ K_0:=\left(\|c\|_{[t_0,\infty)}+\sum_{k=1}^m \|d_k\|_{[t_0,\infty)}\right)
\sum_{k=1}^m \tau_k\left\|\frac{d_k}{d-c}\right\|_{[t_0,\infty)}<1
\end{equation}
then (\ref{3.2}) is uniformly exponentially stable. Moreover, the fundamental function $Y(t,s)$ of (\ref{3.2}) satisfies (\ref{3.1}) with $K= (1-K_0)^{-1}.$
\end{lemma}
\begin{proof}
For brevity of notations, we set $y(t)=Y(t,t_0)$. Then, $y$ satisfies (\ref{3.2}), where the initial value is $y(t_0)=1$, with the zero initial function.
Let $J=[t_0,t_1]$, where $t_1>t_0$ is arbitrary. Equality (\ref{3.2}) implies the estimate
\begin{equation}\label{3.3}
 \|\dot{y}\|_J\leq \left(\|c\|_{[t_0,\infty)}+\sum_{k=1}^m \|d_k\|_{[t_0,\infty)}\right)\|y\|_J.
\end{equation}
Equation (\ref{3.2}) can be written as
$$
\dot{y}(t)=-[d(t)-c(t)]y(t)+\sum_{k=1}^m d_k(t)\int_{h_k(t)}^t \dot{y}(\xi)d\xi.
$$
Integrating, we get for $y$
$$
y(t)=e^{-\int_{t_0}^t [d(\xi)-c(\xi)]d\xi}+\int_{t_0}^t e^{-\int_s^t [d(\xi)-c(\xi)]d\xi}[d(s)-c(s)]
\left [\sum_{k=1}^m \frac{d_k(s)}{d(s)-c(s)}\int_{h_k(s)}^s \dot{y}(\xi)d\xi\right]ds.
$$
Then, by the above equality, (\ref{3.3}) and the definition of $K_0$ in (\ref{3.2a}),
$$
\|y\|_J\leq 1+\sum_{k=1}^m \tau_k\left\|\frac{d_k}{d-c}\right\|_{[t_0,\infty)}\|\dot{y}\|_J
\leq 1+ K_0 \|y\|_J.
%
$$
Then
$ \|Y(t,t_0)\|_J\leq (1-K_0)^{-1}$, and the expression in 
the right-hand side 
does not depend on $t_1$. 

Hence 
 $ \displaystyle \|Y(t,t_0)\|_{[t_0,\infty)}\leq  (1-K_0)^{-1}$.
Again,  the number in the right-hand side 
does not depend on $t_0$. Thus estimate (\ref{3.1}) holds.

To prove exponential stability, we apply the Bohl-Perron theorem. 
Let $f:[t_0,\infty)\to {\mathbb R}$ be globally essentially bounded, and $y$ be a solution of  
\begin{equation}\label{3.5a}
\dot{y}(t)=c(t)y(t)-\sum_{k=1}^m d_k(t)y(h_k(t))+f(t), ~t \geq t_0, ~y(t)=0, t\leq t_0.
\end{equation}
Equality (\ref{3.5a}) implies
\begin{equation}\label{3.6}
 \|\dot{y}\|_J\leq \left(\|c\|_{[t_0,\infty)}+\sum_{k=1}^m \|d_k\|_{[t_0,\infty)}\right)\|y\|_J+\|f\|_{[t_0,\infty)}.
\end{equation}

Similarly to the above argument, equation (\ref{3.5a}) can be rewritten and integrated:
\begin{align*}
\dot{y}(t)= &-[d(t)-c(t)]y(t)+\sum_{k=1}^m d_k(t)\int_{h_k(t)}^t \dot{y}(\xi)d\xi+f(t),
\\
y(t)= &\int_{t_0}^t e^{-\int_s^t [d(\xi)-c(\xi)]d\xi}[d(s)-c(s)]
\left [\sum_{k=1}^m \frac{d_k(s)}{d(s)-c(s)}\int_{h_k(s)}^s \dot{y}(\xi)d\xi\right] \, ds
\\ & +\int_{t_0}^t e^{-\int_s^t [d(\xi)-c(\xi)]d\xi}[d(s)-c(s)]\frac{f(s)}{d(s)-c(s)} \, ds.
\end{align*}
Since $d(s)-c(s)\geq \alpha_0>0$ and $\|f\|_{[t_0,\infty)}<\infty$, we have $\|\frac{f}{d-c}\|_{[t_0,\infty)}<\infty$.
Then, by  (\ref{3.6}), 
$$
\|y\|_J  \leq \sum_{k=1}^m \tau_k\left\|\frac{d_k}{d-c}\right\|_{[t_0,\infty)}\|\dot{y}\|_J +
\left\|\frac{f}{d-c} \right\|_{[t_0,\infty)}
\leq  K_0 \|y\|_J + M_0, 
$$
where
$$
M_0=\sum_{k=1}^m \tau_k\left\|\frac{d_k}{d-c}\right\|_{[t_0,\infty)}
\|f\|_{[t_0,\infty)}+ \left\|\frac{f}{d-c} \right\|_{[t_0,\infty)}.
$$
Hence for the constant $ M=(1-K_0)^{-1} M_0>0$, we have $\|y\|_J\leq M$, where $M$ does not depend on the interval $J$.
Therefore $\|x\|_{[t_0,\infty)}<\infty$.
Appliing Lemma \ref{lemmaBP}, we conclude that (\ref{3.2}) is uniformly exponentially stable.
\end{proof}

\begin{remark}\label{remark3.1}
In Lemma~\ref{lemma3.1}, instead of $K_0$, we can take any number $K_1 \in (K_0,1)$.
\end{remark}

\begin{remark}\label{remark3.2}
Equation (\ref{3.2}) is uniformly exponentially stable
under the assumptions of the lemma.
In contrast to a known result on exponential stability of an equation $\dot{y}+c(t)y(t) + d(t)y(h(t))=0$
with a dominating  non-delay term $c(t) > |d(t)| \geq 0$, we consider the case $c(t)\leq 0$, $d(t) > c(t)$.
In this sense, the delay term is dominating over the non-delay term.
\end{remark}

\section{Main Results}

\subsection{Exponentially stable equations}

We start with an exponential estimate for an exponentially stable equation.

\begin{theorem}\label{theorem4.1}
Assume that there are constants $\lambda>0$ and $\alpha>0$ such that the following inequalities hold:
\begin{equation}\label{4.0}
a(t):=\sum_{k=1}^m e^{\lambda(t-h_k(t))}b_k(t)-\lambda\geq \alpha, \quad t \geq t_0, 
\end{equation}
\begin{equation}\label{4.0a}
M_1:=\left(\lambda+\sum_{k=1}^m e^{\lambda\tau_k}\|b_k\|_{[t_0,\infty)}\right)
\sum_{k=1}^m \tau_k \left\|\frac{a_k}{a}\right\|_{[t_0,\infty)}<1,
\end{equation}
where $a_k(t)=e^{\lambda(t-h_k(t))}b_k(t)$.

Then for the solution of problem ({2.1}),  the following estimate is valid
\begin{equation}\label{4.1}
|x(t)|\leq M_0e^{-\lambda(t-t_0)}\left[|x(t_0)|+\sum_{k=1}^m\frac{e^{\lambda \tau_k}-1}{\lambda} \|b_k\|_{[t_0,\infty)}\|\phi\|_{[t_0-\tau,t_0]}\right]+\frac{M_0}{\lambda}\|f\|_{[t_0,t]}, 
\end{equation}
where $M_0:=(1-M_1)^{-1}$. 
\end{theorem}
\begin{proof}
Consider first the case $f \equiv 0$.
After the substitution $x(t)=e^{-\lambda(t-t_0)}z(t)$ into (\ref{2.1}), we get
\begin{equation}\label{4.3}
\dot{z}(t)=\lambda z(t)-\sum_{k=1}^m e^{\lambda(t-h_k(t))}b_k(t)z(h_k(t)).
\end{equation}
Equation (\ref{4.3}) has the form of (\ref{3.2}) with 
$$
c(t)=\lambda, ~~d_k(t)=e^{\lambda(t-h_k(t))}b_k(t),~~ d(t)=\sum_{k=1}^m e^{\lambda(t-h_k(t))}b_k(t). 
$$
Then in (\ref{4.0}), $a(t)$ corresponds to $d(t)-c(t)$.
Let $Z(t,s)$ be the fundamental function of (\ref{4.3}).
Inequalities (\ref{4.0}) and (\ref{4.0a}) imply  (\ref{3.2a}).  
By Lemma \ref{lemma3.1}, $|Z(t,s)|\leq M_0$. Let $X(t,s)$ be
a fundamental function  of (\ref{2.1}). Then, for $X(t,s)$  we have the exponential equality $X(t,s)=e^{-\lambda(t-s)}Z(t,s)$.
Hence $|X(t,s)| \leq M_0 e^{-\lambda(t-s)}$. By (\ref{2.2}), for the solution $x$ of problem ({2.1}) we have
\begin{align*}
|x(t)| \leq & ~ |X(t,t_0)||x(t_0)|+\sum_{k=1}^m \int_{t_0}^{t_0+\tau_k}|X(t,s)||b_k(s)||\phi(h_k(s))|ds
\\
\leq & ~M_0 e^{-\lambda(t-t_0)} |x(t_0)|+\frac{M_0}{\lambda} 
\sum_{k=1}^m \|b_k\|_{[t_0,\infty)} \left( e^{-\lambda (t-t_0-\tau_k)}
-e^{-\lambda(t-t_0)} \right)\|\phi\|_{[t_0-\tau,t_0]},
\end{align*}
which implies (\ref{4.1}) with $f \equiv 0$.

For the general case we apply (\ref{2.2}), the estimate for $X(t,s)$ and the inequalities
$$
\left| \int_{t_0}^t X(t,s) f(s)\, ds \right| 
\leq M_0 \|f\|_{[t_0,t]} \int_{t_0}^t e^{-\lambda(t-t_0)}ds \leq \frac{M_0}{\lambda} \|f\|_{[t_0,t]}.
$$
\end{proof}

\begin{remark}
Note that if the conditions of Theorem \ref{theorem4.1} hold,  (\ref{1.1}) is uniformly exponentially stable.
\end{remark}

\begin{remark}
Let us apply Theorem \ref{theorem4.1} to a non-delay equation
\begin{equation}\label{4.3a}
\dot{x}(t)+b(t)x(t)=f(t), ~~t\geq t_0,~~ x(t_0)=x_0,
\end{equation}
where $b(t)\geq b_0>0$. If we take $\lambda=b_0$, conditions  (\ref{4.0}) and (\ref{4.0a}) hold, $M_0=1$.
Estimate (\ref{4.1}) becomes
$$
|x(t)|\leq e^{-b_0(t-t_0)}|x_0|+\frac{1}{b_0}\|f\|_{[t_0,t]},
$$
where for constant positive $b$ and $f$ the equality holds.
Hence estimate  (\ref{4.1}) in some sense is sharp.
\end{remark}

The continuity of the functions $u(\lambda)=e^{\lambda \tau_k}$ which tend to one as $\lambda \to 0^+$ implies the following 
result.

\begin{corollary}\label{corollary4.1}
Assume that for some $\lambda_0>0$, $\alpha>0$,
\begin{equation*}
\sum_{k=1}^m e^{\lambda_0(t-h_k(t))}b_k(t)-\lambda_0 \geq \alpha,~~
\sum_{k=1}^m \|b_k\|_{[t_0,\infty)} \sum_{k=1}^m \tau_k\|b_k\|_{[t_0,\infty)} <\alpha.
\end{equation*}
Then equation (\ref{1.1}) is uniformly exponentially stable, and for some $\lambda \in (0,\lambda_0]$,
the solution of problem (\ref{2.1}) satisfies estimate (\ref{4.1}).
\end{corollary}

Theorem~\ref{theorem4.1} and its corollary give only implicit conditions 
for the estimation of solutions of equation (\ref{1.1}).
The following theorem contains explicit conditions for the estimates.

\begin{theorem}\label{theorem4.2}
Assume that
\begin{equation}\label{4_star}
b_k(t)\geq 0, ~~b_0 := \liminf_{t\rightarrow\infty}\sum_{k=1}^m b_k(t)>0, 
~~\sum_{k=1}^m \|b_k\|_{[t_0,\infty)}\sum_{k=1}^m \tau_k\left\|\frac{b_k}{\sum_{i=1}^m b_i}\right\|_{[t_0,\infty)}<1. 
\end{equation}
Then 
there exists a unique solution $\lambda_0\in (0,b_0)$ of the equation
\begin{equation}\label{4c}
\left(\lambda+\sum_{k=1}^m e^{\lambda\tau_k}\|b_k\|_{[t_0,\infty)}\right)
\sum_{k=1}^m \tau_k e^{\lambda\tau_k} \left\|\frac{b_k}{\sum_{k=1}^m b_k-\lambda}\right\|_{[t_0,\infty)}=1.
\end{equation}
If~ $0<\lambda<\lambda_0$, the solution of problem (\ref{2.1}) satisfies (\ref{4.1}). 
\end{theorem}
\begin{proof}
For $0<\lambda<b_0$, where $b_0$ is defined in (\ref{4_star}), denote 
$$
g(\lambda):=\left(\lambda+\sum_{k=1}^m e^{\lambda\tau_k}\|b_k\|_{[t_0,\infty)}\right)
\sum_{k=1}^m \tau_k e^{\lambda\tau_k} \left\|\frac{b_k}{\sum_{k=1}^m b_k -\lambda}\right\|_{[t_0,\infty)}.
$$
This function is well defined as $\sum_{k=1}^m b_k -\lambda > 0$ for $\lambda \in (0,b_0)$, the terms 
$e^{\lambda\tau_k}$ increase in $\lambda$, while $\sum_{k=1}^m b_k -\lambda$ in the denominator is positive and decreasing, leading to the increase in the fraction, thus $g$ is monotone increasing in $\lambda$.
Since by (\ref{4_star}), $g(0)<1$, $g(b_0)=\infty$ and $g$ is monotone increasing, the equation $g(\lambda)=1$ has a unique solution
$\lambda\in (0,b_0)$, 
which is denoted by $\lambda_0$.  By the definition of $g$, $\displaystyle \left|  \sum_{k=1}^m b_k(t)-\lambda_0 \right|
\geq \varepsilon b_k(t)$ for $k=1, \dots, m$ and some $\varepsilon>0$.
Let $\lambda \in (0,\lambda_0)$. Then
$$
\left|\frac{b_k(t)}{\sum_{k=1}^m e^{\lambda(t-h_k(t))}b_k(t)-\lambda}\right|
=\frac{b_k(t)}{\sum_{k=1}^m e^{\lambda(t-h_k(t))}b_k(t)-\lambda}
<\left\|\frac{b_k(t)}{\sum_{k=1}^m b_k(t)-\lambda_0}\right\|_{[t_0,\infty)}.
$$
Hence
$$
\left(\lambda+\sum_{k=1}^m e^{\lambda\tau_k}\|b_k\|_{[t_0,\infty)}\right)
\sum_{k=1}^m \tau_k e^{\lambda\tau_k} \left\|\frac{b_k}{a}\right\|_{[t_0,\infty)}< g(\lambda_0)=1.
$$
By Theorem~\ref{theorem4.1},  the solution of problem (\ref{2.1}) satisfies (\ref{4.1}).
\end{proof}

\begin{corollary}\label{corollary4.3}
Assume that 
$b_k(t)\equiv b_k>0$, $k=1,\dots,m$, 
$\displaystyle b_0 :=\sum_{k=1}^m b_k$ 
and $\displaystyle \sum_{k=1}^m \tau_k b_k<1$.
Then there is a unique solution $\lambda_0\in (0,b_0)$ of 
$$
\left(\lambda+\sum_{k=1}^m e^{\lambda\tau_k}b_k\right)
\sum_{k=1}^m \tau_k e^{\lambda\tau_k} \frac{b_k}{\sum_{k=1}^m b_k-\lambda}=1.
$$
If~ $0<\lambda<\lambda_0$, the solution of problem (\ref{2.1}) satisfies (\ref{4.1}), , where
$\|b_k\|$ is replaced by $b_k$.
\end{corollary}

\subsection{The equation is not exponentially stable}

We proceed to the case when the equation in (\ref{2.1}) is not exponentially stable, or its asymptotic behaviour is unknown.

\begin{theorem}\label{theorem4.3}
Let there exist $\lambda>0$ and $\alpha>0$ such that
\begin{equation}\label{4.4}
a(t):=\sum_{k=1}^m e^{-\lambda(t-h_k(t))}b_k(t)+\lambda\geq \alpha,
\end{equation}
\begin{equation}\label{4.4c}
M_2:=\left(\lambda+\sum_{k=1}^m e^{-\lambda\delta_k}\|b_k\|_{[t_0,\infty)}\right)
\sum_{k=1}^m \tau_k  \left\|\frac{a_k}{a}\right\|_{[t_0,\infty)}<1,
\end{equation}
where $a_k(t)=e^{-\lambda(t-h_k(t))}b_k(t)$.
Then the solution of problem (\ref{2.1})  satisfies
\begin{equation}\label{4.5}
|x(t)|\leq M_0e^{\lambda(t-t_0)}\left[|x(t_0)|+\sum_{k=1}^m\frac{1-e^{-\lambda \tau_k}}{\lambda}
\|b_k\|_{[t_0,\infty)}\|\phi\|_{[t_0-\tau,t_0]}+\frac{\|f\|_{[t_0,t]}}{\lambda}\right],
\end{equation}
where $M_0:=(1-M_2)^{-1}$.
\end{theorem}

\begin{proof}
First, let $f \equiv 0$.
After substituting $x(t)=e^{\lambda(t-t_0)}z(t)$ into (\ref{2.1}), we get
\begin{equation}\label{4.7}
\dot{z}(t)=-\lambda z(t)-\sum_{k=1}^m e^{-\lambda(t-h_k(t))}b_k(t)z(h_k(t)).
\end{equation}
Equation (\ref{4.7}) has the form of (\ref{3.2}), where 
$$
c(t)=-\lambda, ~d_k(t)=e^{-\lambda(t-h_k(t))}b_k(t), ~d(t)=\sum_{k=1}^m e^{-\lambda(t-h_k(t))}b_k(t), ~a(t):=d(t)-c(t).
$$
Let $Z(t,s)$ be the fundamental function of equation (\ref{4.7}).
By (\ref{4.4}),  $d(t)-c(t)\geq \alpha$.

Inequality (\ref{4.4c}) implies 
$$
K_0=\left(\|c\|_{[t_0,\infty)}+\sum_{k=1}^m \|d_k\|_{[t_0,\infty)}\right)
\sum_{k=1}^m \tau_k\left\|\frac{d_k}{d-c}\right\|_{[t_0,\infty)}\leq M_2<1.
$$

By Lemma \ref{lemma3.1} and Remark \ref{remark3.1}, $|Z(t,s)|\leq M_0$.
For the fundamental function  of (\ref{2.1}) we have $X(t,s)=e^{\lambda(t-s)}Z(t,s)$.
Hence $|X(t,s)| \leq M_0 e^{\lambda(t-s)}$. By (\ref{2.2}), the solution $x$ of problem (\ref{2.1}) satisfies
\begin{align*}
|x(t)|\leq & ~ |X(t,t_0)|~|x(t_0)|+\sum_{k=1}^m \int_{t_0}^{t_0+\tau_k}|X(t,s)|~|b_k(s)|~|\phi(h_k(s))| \,ds
\\
\leq & ~ M_0 e^{\lambda(t-t_0)} |x(t_0)| - \frac{M_0}{\lambda} 
\sum_{k=1}^m \|b_k\|_{[t_0,\infty)} \left( e^{\lambda (t-t_0-\tau_k)}
-e^{\lambda(t-t_0)} \right)\|\phi\|_{[t_0-\tau,t_0]},
\end{align*}
which implies (\ref{4.5}) with $f \equiv 0$.
For the general case, we apply 
\begin{align*}
&~\left| \int_{t_0}^t X(t,s) f(s)\, ds \right| \leq \|f\|_{[t_0,t]} \int_{t_0}^t M_0 e^{\lambda(t-s)} ds
\leq \frac{M_0}{\lambda} e^{\lambda (t-t_0)}\|f\|_{[t_0,t]}.
\end{align*}
\end{proof}

\begin{remark}
Note that for $m=1$, $h_1(t)\equiv t$, $f \equiv 0$, $-b_0 \leq b_1(t) \leq 0$ almost everywhere, where $b_0>0$, for a 
solution $x$ of (\ref{2.1}) the inequality 
$$ x(t)=x(t_0) \exp\left\{ -\int_{t_0}^t b_1(s)\,ds \right\} \leq |x(t_0)| e^{b_0(t-t_0)}$$
holds. For $\lambda=-b_0>0$, Theorem~\ref{theorem4.3} implies 
this estimate which is  sharp for $b_1 \equiv -b_0$.  
\end{remark}

\begin{remark}\label{remark5}
By  \cite[Corollary B.1]{ABBD}, for $X(t,s)$, which is a fundamental function of (\ref{1.1}), the inequality
$$
|X(t,s)|\leq \exp \left\{ \int_s^t \sum_{k=1}^m |b_k(\xi)|d\xi \right\}
$$
holds. 
Hence  $\displaystyle \sum_{k=1}^m |b_k(t)|\leq \lambda_0$ implies inequality (\ref{4.5}) for the solution of problem (\ref{2.1}), where $M_0=1$, $\lambda=\lambda_0$. 
Theorem~\ref{theorem4.3} allows to obtain estimates for solutions of problem (\ref{2.1}) with 
smaller than $\lambda_0$ exponents, see Examples~\ref{example2} and \ref{example2a}.
\end{remark}

\begin{corollary}\label{corollary4.1a}
Assume that there are constants $\lambda>0$ and $\alpha>0$ such that
\begin{equation}\label{4.4_1}
a(t):=\sum_{k=1}^m e^{-\lambda(t-h_k(t))}b_k(t)+\lambda\geq \alpha
\end{equation}
and
\begin{equation}\label{4.4c_1}
M_3:=\left(\lambda+\sum_{k=1}^m e^{-\lambda\delta_k}\|b_k\|_{[t_0,\infty)}\right)
\sum_{k=1}^m \tau_k e^{-\lambda\delta_k} \|b_k\|_{[t_0,\infty)}<\alpha.
\end{equation}
Then for the solution of problem (\ref{2.1}),  estimate  (\ref{4.5}) holds,
where $M_0:=\frac{\alpha}{\alpha-M_3}$.
\end{corollary}

\section{Examples}

In this section, we illustrate Theorems~\ref{theorem4.1}, \ref{theorem4.2} and \ref{theorem4.3} with examples. In Example~\ref{example1}, we construct an exponential estimate with a negative exponent using explicit conditions of Theorem~\ref{theorem4.2}.  Examples~\ref{example2} and \ref{example2a} apply Theorem~\ref{theorem4.3} to an unstable equation. 
In Example~\ref{example3}, we consider a delay equation  with oscillating coefficients and two terms. 
We apply Theorem~\ref{theorem4.1} to establish uniform exponential stability of this equation,
and construct an exponential estimate for solutions of the equation.

\begin{example}
\label{example1}
For the non-homogeneous equation with one delay
\begin{equation}\label{ex1eq1}
\dot{x}(t)+  0.2 (2-\sin t)x \left( t- \frac{2-\sin t}{12} \right) =f(t),~t\geq 0,
\end{equation}
we calculate
$$
\tau = \frac{1}{12} \max_t (2-\sin t)=0.25, ~~\| b\|_{[0,\infty)}= 0.2 \max_t (2-\sin t)= 0.6.
$$
We will apply here Theorem~\ref{theorem4.2}. 

To find $\lambda_0$ we first consider the function
$w(t)=\frac{b(t)}{b(t)-\lambda}=\frac{0.2(2-\sin t)}{0.2(2-\sin t)-\lambda}$, where $0<\lambda<0.2$.
The critical values of this function are solutions of the equation $\cos t=0$. Hence 
$\|w\|_{[t_0,\infty)}=w(\frac{\pi}{2})=\frac{1}{1-5\lambda}$.

Equation (\ref{4c}) for (\ref{ex1eq1}) has the form
$$
\left(\lambda+0.6 e^{\lambda/4}\right)\frac{1}{4} e^{\lambda/4} \, \frac{1}{1-5\lambda}=1.
$$
Numerically we compute the solution of this equation $\lambda_0 \approx 0.159229$, 
the assumptions of Theorem~\ref{theorem4.2} hold for
$\lambda<0.159229$, say, for $\lambda=0.15<\lambda_0$.
It is possible to check numerically that for $\lambda=0.15<\lambda_0$,
$$ a(t) := 0.2 (2-\sin t) e^{(2-\sin t) \lambda/12} - \lambda \geq \alpha > 0.05,$$
$$
\left\| \frac{b}{a}\right\|_{[0,\infty)}= \left\| \frac{ 0.2(2-\sin t)}{0.2 (2-\sin t) e^{(2-\sin t) \lambda/12} - 
\lambda}
\right\|_{[0,\infty)} < 3.85 < 4.
$$
Since
$\displaystyle M_1=
\left( \lambda+e^{\lambda \tau} \| b\|_{[0,\infty)} \right) \tau e^{\lambda \tau} \left\| 
\frac{b}{a}\right\|_{[0,\infty)} \leq
\left( 0.15+ 0.6 e^{0.15 \cdot 0.25} \right) 0.25 e^{0.15 \cdot 0.25} \cdot 4 < 0.803 <1,
$
the assumptions of Theorem~\ref{theorem4.1} are also satisfied, $M_0<1/(1-0.803)<5.1$, and exponential estimates 
imply. We have $\displaystyle \| b\|_{[0,\infty)} (e^{\lambda \tau}-1)/\lambda<0.153$,
thus
the solution of (\ref{ex1eq1}) with $x(0)=1$,     
$\phi(t)=\sin(20 t)-1$, $\|\phi\|_{[-0.25,0]}=2$, $f\equiv 0$
satisfies $$ |x(t)| \leq 5.1  e^{-0.15 t} \big[ |x(0)| + 0.153  \|\phi\|_{[-0.25,0]}  \big]
\leq 5.1 \cdot 1.306 e^{-0.15 t} \leq 6.67 e^{-0.15 t}.$$
For  $f(t) \equiv 0$,
the comparison with the numerical solution is illustrated in Fig.~\ref{figure1}, left.

Next, apply Theorem~\ref{theorem4.1}, where the right-hand side is $f(t)=0.05 (1+\sin(2t))$, $\|f \|_{[0,\infty)}=0.1$,
$M_0/\lambda \|f \|_{[0,\infty)} \leq 0.1~5.1/0.15= 3.4$, and we get $\displaystyle |x(t)| \leq 6.67 e^{-0.15 t}+ 
3.4$. See Fig.~\ref{figure1}, right, for the comparison with numerical results.


\begin{figure}[ht]
\centering
\includegraphics[scale=0.4]{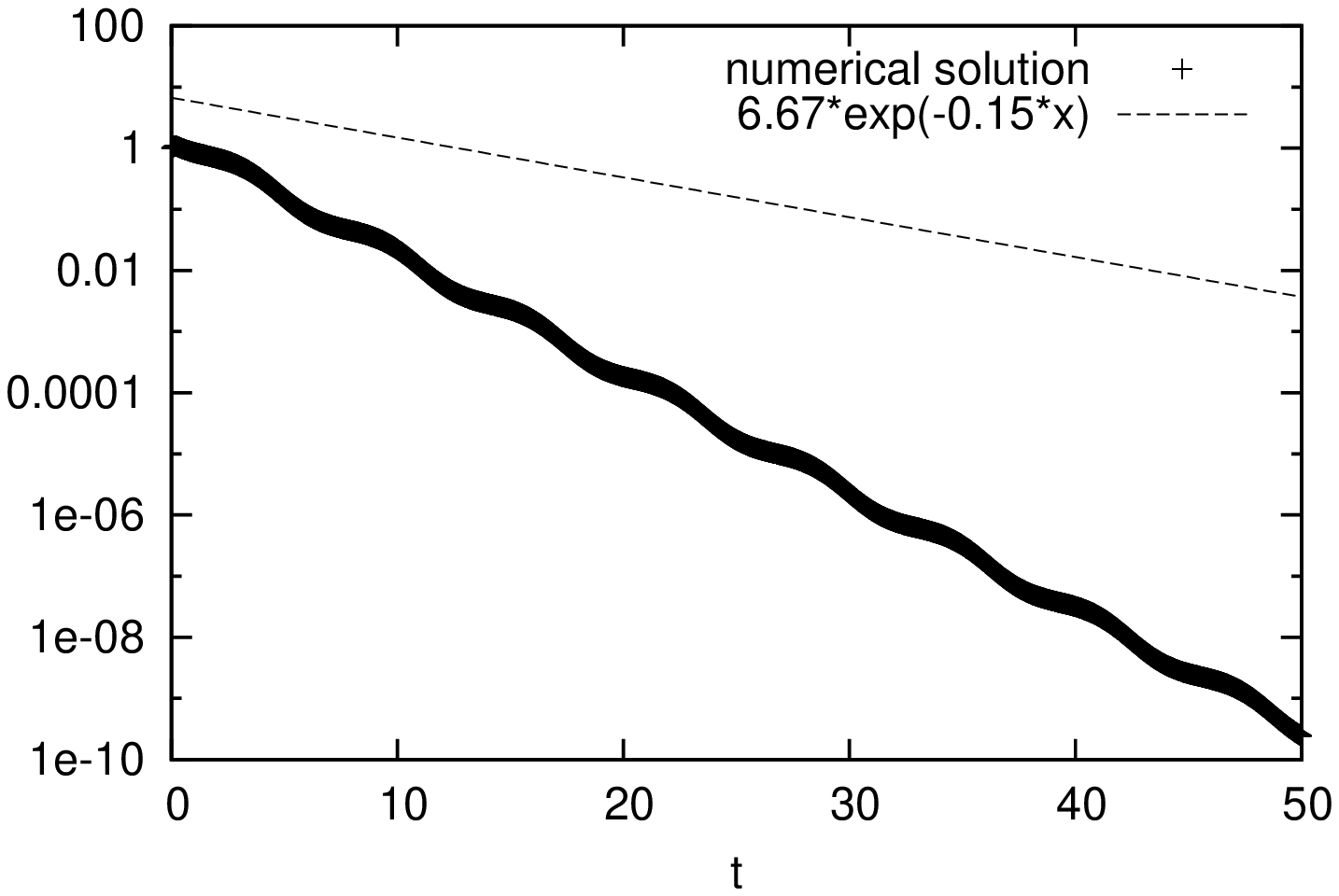}~~~~~~~~~
\includegraphics[scale=0.4]{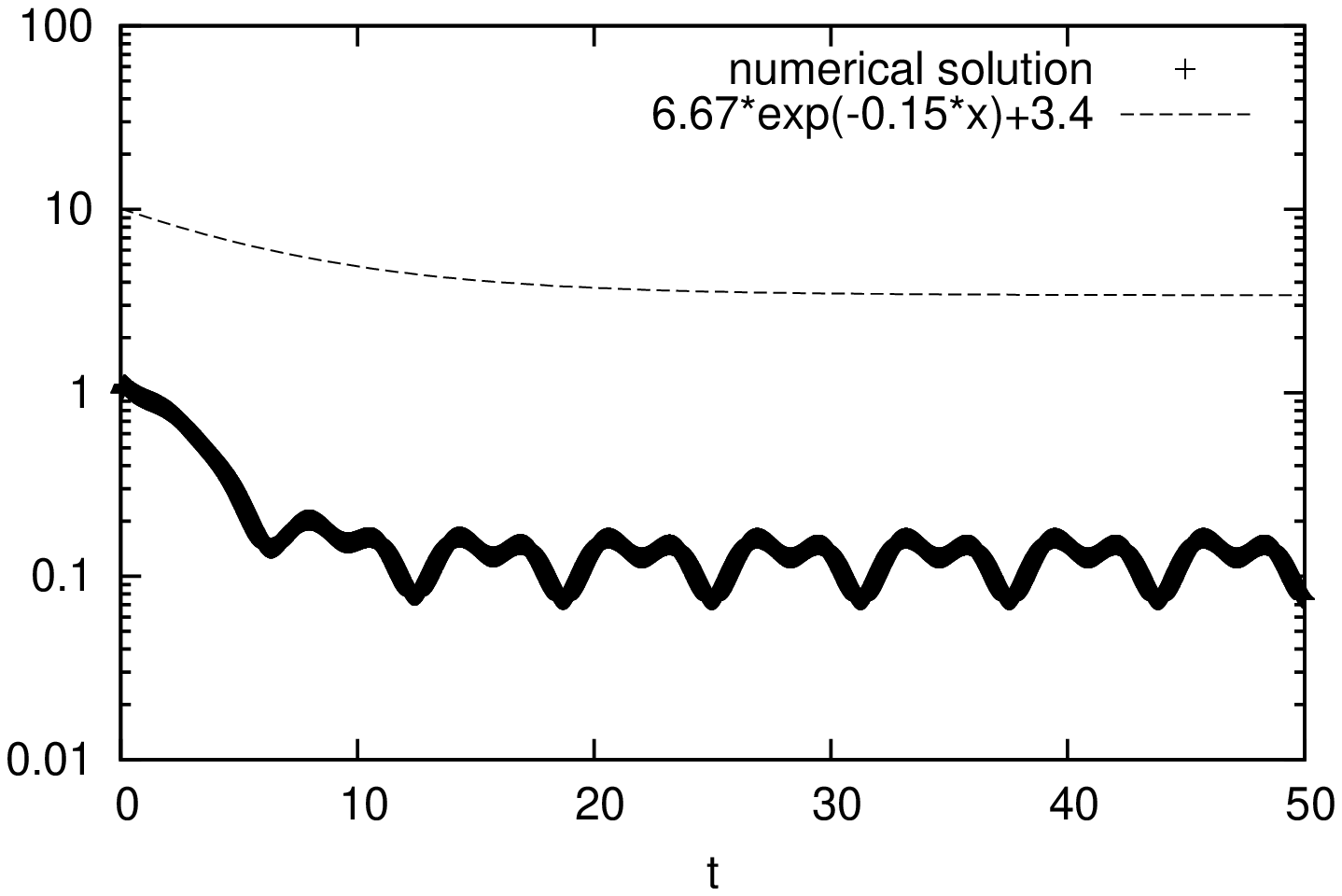}
\caption{The absolute value of the numerical solution of (\protect{\ref{ex1eq1}}) with $x(0)=1$, 
$\phi(t)=\sin(20 t)-1$ 
compared to the theoretical estimate from Theorem~\protect{\ref{theorem4.1}} 
for (left) $f\equiv 0$; (right) $f(t)=0.05 (1+\sin(2t))$. The scale in
$x$ is logarithmic.}
\label{figure1}
\end{figure}  
\end{example}


\begin{example}
\label{example2}
As it is well known, the autonomous equation 
\begin{equation}
\label{ex_add_0}
\dot{x}(t)+2x(t-1)=0
\end{equation}
is unstable.
Let $a(t)=\lambda+2e^{-\lambda}=\alpha$.
The assumptions of Corollary~\ref{corollary4.1a} are satisfied whenever
$$\frac{M_2}{\alpha}=\left( \lambda+2e^{-\lambda} \right) e^{-\lambda} \frac{2}{\lambda+2e^{-\lambda}}=2e^{-\lambda} 
<1,$$
or $\lambda > \ln 2 \approx  0.693147$. Take $\lambda= 0.8 > \ln 2$, then $\alpha=2e^{-\lambda} + \lambda > 1.6986$, 
$M_2=2(\lambda+2e^{-\lambda})e^{-\lambda}<1.527<\alpha$, $\displaystyle M_0=\alpha/(\alpha-M_2)<10$. In
particular, $|x(t)|
\leq 10 e^{0.8 t}|x(0)|$ for the zero initial function.
By Remark~\ref{remark5}, we can get  $|x(t)| \leq e^{2 t}$, which is a less sharp estimate for any $t \geq 1.92$.
\end{example}

\begin{example}
\label{example2a}
Further, consider the problem 
\begin{equation}\label{ex_add_1}
\dot{x}(t)+  0.2 (2-\sin t)x \left( t- 15 + \sin t \right) = 0,~t\geq 0, ~x(t)=0, ~t<0, x(0)=1.
\end{equation}
Here 
$\| b\|_{[0,\infty)}= 0.6$, $\tau=16$, $\delta=14$, 
take $\lambda=0.2$. Then 
$a(t)=e^{-0.2(15-\sin t)} 0.2 (2-\sin t)+ 0.2>0.2=\alpha$,
we evaluate numerically $\|b/a\|_{[0,\infty)}  < 2.6735$. Thus
$$
M_2= 16 e^{-2.8} \left\| \frac{b}{a} \right\|_{[0,\infty)} \left( 0.2+0.6e^{-0.2\cdot 14}\right) \approx 0.61515 <1. 
$$
Then
$M_0<2.6$, and we get the estimate 
$\displaystyle
|x(t)| <2.6 e^{0.2 t}$
for the solution of (\ref{ex_add_1}), see Fig.~\ref{figure2a} for the comparison with the absolute value of the 
numerical solution.

By Remark~\ref{remark5}, we get  $\displaystyle
|x(t)| < e^{0.6 t}$, which is less sharp for $t \geq 2.389$.

For the chosen initial conditions, the exponential estimate computed numerically has $\lambda \approx 0.09$.
However, higher growth rates are possible for the same or even smaller delays and the same coefficient. 
Let $\lfloor t \rfloor$ be the maximal integer not exceeding $t$.
In the problem
\begin{equation}\label{ex_add_1a}
\dot{x}(t)+  0.2 (2-\sin t)x (h(t)) = 0,~t\geq 0, ~x(t)=0, ~t<0, x(0)=1,~h(t)= 4\pi \left\lfloor \frac{t}{4\pi} \right\rfloor,
\end{equation}
the maximal delay of $4\pi$ is less than 14.
Direct computation implies $$x(t)=1-\int_0^t 0.2(2-\sin s)~ds,~ s\in [0,4\pi],~~
x(4\pi) =1 - 0.2 \int_0^{4\pi}(2-\sin s)~ds=1-1.6 \pi<0. $$
Next, $x(8\pi)= (1-1.6 \pi)-(1-1.6 \pi)1.6\pi=(1-1.6 \pi)^2$, $x(4\pi k)=(1-1.6 \pi)^k$.
Here the growth coefficient $\lambda = \ln(1.6\pi-1)/(4\pi) \approx 0.11$.



\begin{figure}[ht]
\centering
\includegraphics[scale=0.4]{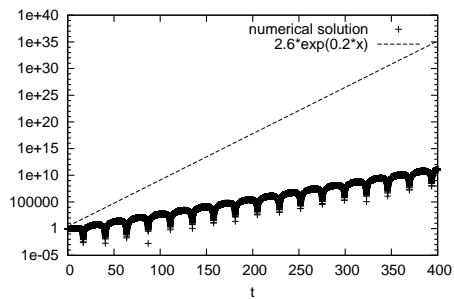}
\caption{The absolute value of the solution of (\protect{\ref{ex_add_1}}) with $x(0)=1$,
$\phi(t)=0$
compared to the theoretical estimate from Theorem~\protect{\ref{theorem4.3}}.
The scale in
$x$ is logarithmic.}
\label{figure2a}
\end{figure}
\end{example}
For both (\ref{ex_add_0}) and (\ref{ex_add_1}), Theorem~\ref{theorem4.3} gives an estimate with a lower exponent than 
Remark~\ref{remark5}. 


Finally, consider an example with oscillating coefficients. 

\begin{example}
\label{example3}
For the equation
\begin{equation}\label{ex1eq3}
\dot{x}(t)+  0.2 (0.5-\sin (t))x \left( t- \frac{2-\sin t}{6} \right) 
+ 0.2 (0.5+\sin (t))x \left( t- \frac{2-\sin t}{12} \right)= 0,~t\geq 0,
\end{equation}
we have $\tau_1=0.5$, $\tau_2 = 0.25$, $\| b_1\|_{[0,\infty)}=\| b_2\|_{[0,\infty)}=0.3$. 
Take $\lambda=0.02$, then we estimate numerically that
$$ a(t) \geq \alpha \approx 0.18,~\|b_1/a\|_{[0,\infty)} < 1.691,~\|b_2/a\|_{[0,\infty)} < 1.669,$$
in (\ref{4.0a}) we have 
$$M_1< \left( 0.02+0.3 e^{0.01}+0.3e^{0.005}  \right)\left(0.5 e^{0.01}\cdot 1.691+0.25e^{0.005}\cdot 1.669  \right) 
< 0.7953,$$
and $M_0 <4.89$, thus
for $x(0)=1$, $\phi(t) \equiv 0$, the solution of (\ref{ex1eq3}) satisfies 
\begin{equation}\label{osc}
 |x(t)| \leq 4.89 e^{-0.02 t}.
\end{equation}
\end{example}

\begin{remark}
Inequality \eqref{osc} implies that equation \eqref{ex1eq3} with two oscillating coefficients is uniformly exponentially stable.
We do not know other stability tests for this class of equations.
\end{remark}

\section{Conclusions and Discussion}

As we mentioned in Introduction, there are several related topics, in particular, asymptotic integration
or obtaining  asymptotic formulas for solutions of a given
equation. Such asymptotic expressions of the form
$\displaystyle
x(t)=(C+o(1))e^{\alpha t} $
have been constructed for various classes of differential equations:
scalar and vector, linear and nonlinear,  
equations with concentrated and distributed delays, including integro-differential equations. 
The review paper \cite{Appleby2} includes a detailed outline of these results with historical notes. 
Asymptotic formulas give a good approximation for long time behavior
of solutions, but are useless for finite intervals. 
Estimates obtained in the present paper apply to finite intervals, and constitute a natural addition to asymptotic formulas.
Most papers on asymptotic integration consider specific classes of equations, such as equations with
a dominating non-delay term, autonomous equations, equations with a single delay.
We have considered scalar equations
with several variable delays and measurable coefficients. 
Another advantage of our results is an explicit
form for these estimates, which we illustrated by examples of different types.

When investigating asymptotic stability, most authors considered equations with positive coefficients.
If both positive and negative coefficients were involved, it was usually assumed that a positive part dominates over the negative one, 
see, for example, \cite{AS,BB2,BB3,BB1,GD,GH2,Kz}. 
In Theorem~\ref{theorem4.1}, we got uniform exponential stability conditions, when the equation contains 
several delays, and all the coefficients may be oscillatory, see Example \ref{ex1eq3}. 
We suggest that in this case, not only the estimate, but also the stability test is new.

Assume that $b_k$ in Theorem~\ref{theorem4.2}  are proportional. Then (\ref{4_star}) leads to the following uniform exponential stability condition
\begin{equation}
\label{condition}
0 \leq b_k(t)\leq \beta_k, ~
\liminf_{t\rightarrow\infty}\sum_{k=1}^m b_k(t)>0,~t-\tau_k\leq h_k(t) \leq t,
~k=1, \dots,m,~\sum_{k=1}^m b_k  \tau_k < 1
\end{equation}
for (\ref{2.1}), and to a relevant estimate of solutions. Note that for (\ref{2.1}) with several variable delays and 
coefficients, one in the right-hand side of  (\ref{condition}) is the best possible constant \cite{Kz}.

Let us state possible extensions of the results of the present paper.

\begin{enumerate}
\item
Obtain explicit estimates of solutions for nonlinear equations in general, and for
specific cases, for example, models of population dynamics. 
These include, but not limited to,  the Hutchinson, the Mackey-Glass  and the Nicholson blowflies equations.
\item
Extend the estimates to a vector DDE. Consider other types of delay, such as distributed, as well as equations of the second and higher order, and stochastic differential equations.
\item
In this paper, we presented pointwise estimates. It would be interesting to obtain estimates in an integral form.
\item
For the case of a single delay
$$
\dot{x}(t)+b(t)x(h(t))=f(t), ~~b(t)\geq b_0>0, t-h(t)\leq \tau,
$$
we obtain estimates 
for $\|b\|\tau <1$. Relax this condition, replacing $1$ with either
$\frac{3}{2}$  or a sharp constant in $(1,1.5)$.
\end{enumerate}

\section*{Acknowledgment}

The second author acknowledges the support of NSERC, the grant RGPIN-2015-05976.
Both authors are grateful to the reviewer for valuable comments and suggestions.

\end{document}